\documentclass[11pt]{amsart}

\usepackage{amsthm}
\usepackage{amssymb}
\usepackage[all]{xy}
\usepackage[colorlinks=true, urlcolor=rltblue, citecolor=drkgreen, linkcolor=drkred] {hyperref}
\usepackage{color}
\usepackage{pdfsync}
\definecolor{rltblue}{rgb}{0,0,0.4}
\definecolor{drkgreen}{rgb}{0,0.4,0}
\definecolor{drkred}{rgb}{0.5,0,0}

\newtheorem{thm}{Theorem}[section]
\newtheorem{lemma}[thm]{Lemma}
\newtheorem{proposition}[thm]{Proposition}
\newtheorem{cor}[thm]{Corollary}
\newtheorem{theorem}[thm]{Theorem}
\newtheorem{corollary}[thm]{Corollary}

\theoremstyle{definition}
\newtheorem{definition}[thm]{Definition}
\newtheorem{exmp}[thm]{Example}

\theoremstyle{remark}
\newtheorem{observation}[thm]{Observation}

\newtheorem{remark}[thm]{Remark}

\newtheorem{historic}[thm]{Historic Remark}

\newtheorem{question}{Question}

\theoremstyle{plain}

\newtheorem*{mainquestion}{Main Question}

\newtheorem{prop}[thm]{Proposition}

\newtheorem{lem}[thm]{Lemma}

\newcounter{contenumi}

\def\D{{\mathcal D}}

\def\dprime{^{\prime\prime}}

\def\la{\langle}
\def\ra{\rangle}
\def\and{\mathrel{\&}}

\def\isom{\cong}
\def\Si{\Sigma}

\newcommand\rightdate[1]{\footnotetext{  Saved: #1 \\ Compiled: \today}}
\def\A{\mathcal{A}}
\def\B{\mathcal{B}}
\def\C{\mathcal{C}}

\def\F{\mathcal{F}}

\def\om{\omega}
\def\bbar{\bar{b}}

\def\b{\beta}

\def\a{\alpha}
\def\b{\beta}

\def\P{\mathop{\mathcal P}}

\def\QQ{{\mathbb Q}}
\def\M{{\mathcal M}}

\def\A{\mathcal A}

\def\W{\mathcal W}

\def\L{\mathcal{L}}

\def\abar{\bar{a}}
\def\bbar{\bar{b}}

\def\KK{{\mathbb K}}

\def\RR{{\mathbb R}}

\def\Lhat{{\hat{\L}}}
\def\Ahat{\hat{\A}}

\def\g{\gamma}

\def\R{\mathtt{R}}

\newcommand{\revmathfont}[1]{{\textsf{#1}}}

\def\Si{\Sigma}

\def\ZFC{\revmathfont{ZFC}}

\def\om{\omega}

\def\N{\mathcal N}

\def\L{\mathcal{L}}

\def\PP{\mathbb{P}}

\def\gl{\dot{g_0}}
\def\gr{\dot{g_1}}

\def\implies{\Rightarrow}

\def\bfK{{\bf K}}
\def\Lan{{\mathfrak L}}

\def\lemuch{\le_w}
\def\legen{\le_w^*}

\def\equivmuch{\equiv_w}
\def\equivgen{\equivmuch^*}
\def\equivin{\equiv_{\infty\omega}}

\def\Vtt{{\tt{V}}}

\begin{document}


\title{Computable structures in generic extensions}

\author{Julia Knight}





\author{Antonio Montalb\'an}





\author{Noah Schweber}




\address{Department of Mathematics\\
University of Notre Dame\\
 USA}
\email{Julia.F.Knight.1@nd.edu}
\urladdr{\href{http://math.nd.edu/people/faculty/julia-f-knight/}{http://math.nd.edu/people/faculty/julia-f-knight/}}

\address{Department of Mathematics\\
University of California, Berkeley\\
 USA}
\email{antonio@math.berkeley.edu}
\urladdr{\href{http://www.math.berkeley.edu/~antonio/index.html}{www.math.berkeley.edu/$\sim$antonio}}
 
\address{Department of Mathematics\\
University of California, Berkeley\\
 USA}
\email{schweber@math.berkeley.edu}
\urladdr{\href{http://www.math.berkeley.edu/~schweber}{http://www.math.berkeley.edu/$\sim$schweber}}

\thanks{The second author was partially supported by the Packard Fellowship. The third author was partially supported by the Sciobereti Fellowship.}

\rightdate{September 23, 2014. Submitted.}      
\maketitle


%

\begin{abstract} In this paper, we investigate connections between structures present in every generic extension of the universe $V$ and computability theory. We introduce the notion of {\em generic Muchnik reducibility} that can be used to to compare the complexity of uncountable structures; we establish basic properties of this reducibility, and study it in the context of {\em generic presentability}, the existence of a copy of the structure in every extension by a given forcing.  We show that every forcing notion making $\om_2$ countable generically presents some countable structure with no copy in the ground model; and that every structure generically presentable by a forcing notion that does not make $\om_2$ countable has a copy in the ground model.  We also show that any countable structure $\mathcal{A}$ that is generically presentable by a forcing notion not collapsing $\om_1$ has a countable copy in $V$, as does any structure $\mathcal{B}$ generically Muchnik reducible to a structure $\mathcal{A}$ of cardinality $\aleph_1$.  The former positive result yields a new proof of Harrington's result that counterexamples to Vaught's conjecture have models of power $\aleph_1$ with Scott rank arbitrarily high below $\om_2$. Finally, we show that a rigid structure with copies in all generic extensions by a given forcing has a copy already in the ground model.
\end{abstract}

\section{Introduction}

In computable structure theory, one studies the complexity of structures using techniques from computability theory. Almost all of this work concerns countable structures; much less is known about the complexity of uncountable structures. However, the computability theory of uncountable structures has received more attention in the last few years. (See for instance the proceedings volume of the conference Effective Mathematics of the Uncountable \cite{EMU}.)  One idea for studying the complexity of an uncountable structure that seems new is to consider what happens to the structure when its domain is made countable.

Before making this idea more concrete, we recall the notion of \emph{Muchnik reducibility} between countable structures.
This is the standard way in computable structure theory to say that one structure is more complicated than another, in the sense that it harder to compute.

\begin{definition}
Given countable structures $\A$ and $\B$ we say that $\A$ is {\em Muchnik reducible} to $\B$, and we write $\mathcal{A}\lemuch\mathcal{B}$, if, from any copy of $\B$, we can compute a copy of $\A$.
\end{definition}

On its face, this notion is limited to countable structures.  However, by examining generic extensions of the set-theoretic universe, $V$, we can extend it further:

\begin{definition}[Schweber]
For a pair of structures $\A$ and $\B$, not necessarily countable in $V$, we say that $\A$ is {\em generically Muchnik reducible} to $\B$, and we and write 
$\mathcal{A}\legen\mathcal{B}$, if for any generic extension $V[G]$ of the set theoretic universe $V$ in which both structures are countable, we have 
\[V[G]\models\text{$\A\lemuch\B$}. \] 
\end{definition}

In Section \ref{Basic}, we will prove the basic properties of this reducibility. 
We will show that it coincides with Muchnik reducibility on countable structures; i.e., if $\A$ and $\B$ are countable, then $\A\lemuch\B$ if and only if $\A\legen\B$ (Corollary \ref{for: much vs gen}).
Another important fact is that we do not need to consider all the generic extensions that make $\A$ and $\B$ countable.  We will prove that for any two such generic extensions if $\A\lemuch\B$ holds in one, then it holds in the other (Lemma \ref{UNI}).  This shows that generic Muchnik reducibility is a very absolute, and hence, natural, notion of computability-theoretic complexity. 

We will also show that the equivalence $\equivgen$, induced from the reducibility $\legen$, respects $\L_{\infty\om}$-elementary equivalence. In Section \ref{Examples}, we will also exhibit some examples of this reducibility.  For instance, we show that the countable structures generically Muchnik reducible to the linear order $\om_1$ are precisely those Muchnik reducible to some countable well-ordering, and that there are two natural structures, with every countable structure generically reducible to each, one of which is strictly below the other under $\legen$.

Closely related to generic reducibility is \it generic presentability\rm. 

\begin{definition} 
Let $\PP $ be a poset in $V$ and let $\A$ be a structure that lives and is countable in a model, $W$, of $\ZFC$ extending $V$. 
The structure $\A$ is {\em generically presented by $\PP $} if for any $G$ that is $\PP $-generic over $V$, there is a copy $\B$ of $\mathcal{A}$ in $V[G]$ with domain $\omega$. (Here, $\B$ is a {\em copy} of $\A$ in the sense that there is an isomorphism between them in $W[G]$.)
A structure is {\em generically presented} if it is generically presented by some poset $\PP \in V$.  
\end{definition}


\begin{remark}\label{remark on copies}
Throughout this paper we will be considering structures like $\A$ above that do not necessarily live in $V$, but in some extension of $W$.
This is, of course, the point of generic presentability: to study structures which do not live in $V$, but which $V$ can nonetheless talk about to a high level of precision. Once we know that $\A$ is generically presentable, we can replace it by its copy inside one of the generic extensions $V[G]$, and, hence, we could have assumed $W$ was $V[G]$ to begin with.
In the definition above, we only require the isomorphism between $\A$ and $\B$ to be in $W[G]$; this is because it only makes sense to talk about isomorphisms between copies that live in the same model, and we have $\A\in W$ and $\B\in V[G]$.
We will assume this whenever we use the word ``copy'' in the rest of the paper.

A different point of view some readers might prefer is to consider countable models $M$ of $\ZFC$ and define what it means for a countable structure $\A\in V$ to be {\em generically presented by $\PP $ over $M$}.
In this case, since $M$ is countable, it is enough to consider generic extensions $M[G]$ that are still included in $V$, and when we refer to isomorphisms between structures, we just mean isomorphism in $V$.

For clarity, we will use the term {\em $\omega$-copy} to mean a copy $\mathcal{B}$ of a structure $\mathcal{A}\in V$ with domain $\omega$ which may live in a larger set-theoretic universe. So, for example, the field of real numbers has an $\omega$-copy living in any generic extension where $(2^{\aleph_0})^V$ is countable.
\end{remark}

It is well-known that if a set $S$ is in $V[G]$ for every $\PP$-generic $G$, then $S$ must belong to $V$ already (Solovay \cite{Sol70}). However, the situation for isomorphic copies of a given structure is more complicated.  There are cases in which the analogous fact is true, and there are cases in which it is not. This paper is devoted to analyzing this situation, and its connection to uncountable computability. In particular, if $\mathcal{A}\le_w^*\mathcal{B}$, then $\mathcal{B}$ contains all the information necessary to build $\mathcal{A}$ --- up to a certain amount of genericity. To what extent is this genericity actually necessary? Ted Slaman formulated this question as follows:

\begin{mainquestion}[Slaman]
\label{Slaman} 
Let $\A$ and $\B$ be structures, and suppose $\B$ is in G\"odel's $L$.
Does $\mathcal{A}\legen\mathcal{B}$ imply that there is a copy of $\A$ in $L$? 
\end{mainquestion}

(Note that the isomorphism between $\A$ and its hypothetical copy is not required to live in~$\L$.)

We begin by studying the role of forcing-theoretic properties in generic presentability. We prove:

\begin{theorem} \label{notcollapsingom2}
Let $\mathcal{A}$ be a structure that that lives and is countable in an extension $W$ of $V$.
If $\A$ is generically presentable by a forcing notion that does not make $\om_2$ countable, then $\mathcal{A}$ has a copy in $V$.
\end{theorem}

(Note again that, even if the theorem claims the existence of $\B\in V$ isomorphic to $\A$, the isomorphism between $\A$ and $\B$ is only required to belong to $W$, which is where $\A$ lives.
The copy $\B$ given by the theorem might not be countable in $V$, but it will be countable in $W$.
See Remark \ref{remark on copies}.)

This theorem yields as a corollary a partial positive answer to Slaman's question. 

\begin{corollary}\label{cor: Slaman's question}
Suppose $\B\in V$ and $\A$ belongs to, and is countable in, an extension $W$ of $V$.
If $\A\legen\B$ for some $\B\in V$ with $\vert\B\vert=\aleph_1$, then $V$ contains a copy of $\mathcal{A}$.  
\end{corollary}

We can apply the corollary above to  $L$ and get that if $\B$ lives in $L$ and, within $L$, has size $\aleph_1^L$, then $\A$ has a copy in $L$.

We also give a new proof of the following result of Harrington.

\begin{thm} [Harrington]
If $T$ is a counter-example to Vaught's conjecture, then it has models of arbitrarily high Scott rank below $\om_2$.
\end{thm}

On the other hand, making $\om_2$ countable always introduces structures with universe $\om$ that do not have copies in $V$. This provides an exact dichotomy among structures generically presentable, and a negative answer to Slaman's question in general.   

\begin{theorem}\label{thm: omega 2}
There is a structure $\M$ in an extension of $V$ such that $\M$ is generically presentable by any notion of forcing that makes $\omega_2$ countable, but $\M$ has no copy in $V$. Moreover, this $\mathcal{M}$ is generically Muchnik reducible to the ordering $(\omega_2, <)$.
\end{theorem}

%
%
%
%

We close with a structural approach to the question: what properties ensure that generic presentability implies existence in the ground model?  We show that this occurs at least when the structures involved are as ``set-like" as possible. In Section \ref{presentrigid}, we show the following.

\begin{theorem} 
Suppose $\A$ is rigid and is generically presentable. Then there is an isomorphic copy of $\A$ already in $V$.
\end{theorem}


\section{Generic reducibility}  \label{sec: gen red}   


\subsection{Basic properties}\label{Basic} The key result for analyzing generic presentability and generic reducibility is the Shoenfield Absoluteness Theorem (see \cite{Jec03}).  The version we state below is slightly weaker than the actual theorem, but it is all we will need here:

\begin{thm}[Shoenfield] \label{Shoenfield} Suppose $\varphi$ is a $\Pi^1_2$ sentence, with real parameters.
Then, for every forcing extension $W$ of $V$, $V\models \varphi\iff W\models \varphi$.
\end{thm}

An easy fact about (countable) Muchnik reducibility of structures is the following.

\begin{observation}\label{obsv} 
Basic facts about $\lemuch$ are invariant under forcing. Specifically, we have the following.  
\begin{enumerate}
 \item The relation ``$\lemuch$" is $\Pi^1_2$.
\item For countable $\mathcal{A}$, the predicate ``$\ge_w\mathcal{A}$" is $\Delta^1_1$ in a Scott sentence of $\mathcal{A}$.
\end{enumerate}
\end{observation}

Together with Theorem \ref{Shoenfield}, this implies that much of the theory of $\legen$ is absolute. In particular, we have the next lemma.

\begin{lemma}\label{UNI} Fix arbitrary structures $\mathcal{M}, \mathcal{N}$ in $V$. If there is some generic extension in which $\mathcal{M}$ and $\mathcal{N}$ are countable and $\mathcal{M}\lemuch\mathcal{N}$, then $\mathcal{M}\le_w^*\mathcal{N}$.
\end{lemma}

\begin{proof}
Suppose otherwise. Then there must exist posets $\PP _0$ and $\PP _1$ in $V$ such that forcing with either collapses both $\mathcal{M}$ and $\mathcal{N}$, 
 \[
 \Vdash_{\PP _0}\mathcal{M}\lemuch\mathcal{N} 
      \quad \mbox{ and } \quad 
 \Vdash_{\PP _1}\mathcal{M}\not\lemuch\mathcal{N}.
 \] 
 Let $G=H_0\times H_1$ be $\PP _0\times\PP _1$-generic over $V$.
 Let $\M_0$ and $\N_0$ be reals in $V[H_0]$ coding copies of $\M$ and $\N$ with domain $\om$, and let $\M_1$ and $\N_1$ be reals in $V[H_1]$ coding copies of $\M$ and $\N$ with domain $\om$.
 Then, in $V[H_0]$, $\M_0\lemuch \N_0$, while in $V[H_1]$, $\M_1\not\lemuch \N_1$.
By Shoenfield's absoluteness, this is still true in $V[H_0][H_1]$.
This gives us a contradiction because, in $V[H_0][H_1]$,  $\M_0$ is isomorphic to $\M_1$ and $\N_0$ to $\N_1$. 
\end{proof}

As an immediate corollary of Lemma \ref{UNI}, we get the following.  

\begin{corollary} \label{for: much vs gen}
For structures $\mathcal{A}, \mathcal{B}$ countable in $V$, we have $\mathcal{A}\lemuch\mathcal{B}$ if and only if $\mathcal{A}\legen\mathcal{B}$.
\end{corollary}

\subsection{Potential isomorphism}
Generic Muchnik reducibility also has strong connections with infinitary logic.

\begin{definition}  Let $\mathcal{L}$ be a language; that is, a set of relation and operation symbols.  

\begin{itemize}

\item  $\mathcal{L}_{\infty\om}$ is the collection of formulas obtained from the atomic $\mathcal{L}$-formulas by closing under arbitrary set-sized Boolean combinations and single instances of quantification. See \cite{Kei71} for a treatment of the basic properties of $\mathcal{L}_{\infty\om}$.

\item  For structures $\A, \B$ of arbitrary cardinality, we say that $\A$ is {\em $\L_{\infty\om}$-elementary equivalent} to $\B$, and we write $\A\equiv_{\infty\om}\B$, if the structures satisfy the same $\L_{\infty\om}$ sentences. 
\end{itemize}
\end{definition}

There is a structural characterization of $\equiv_{\infty\om}$, due to Carol Karp:

\begin{definition}\label{def: bf system} Suppose $I$ is a set of partial maps. 
We say that $I$ has the {\em back-and-forth property} if $\la\emptyset,\emptyset\ra\in I$ and for every $\la\abar,\bbar\ra\in I$, 
\begin{enumerate}
\item $\abar$ and $\bbar$ satisfy the same atomic formulas,
\item  for every $c\in\A$, there is $d\in\B$ such that $(\abar c, \bbar d)\in I$, and 
\item  for every $d\in\B$, there is $c\in\A$ such that $(\abar c, \bbar d)\in I$.
\end{enumerate}
An $I$ with the back-and-forth property is called a {\em back-and-forth system} between $\mathcal{A}$ and $\mathcal{B}$.
\end{definition}

\begin{thm}[\cite{Karp65}] $\mathcal{A}\equiv_{\infty\om}\mathcal{B}$ iff there is a back-and-forth system between $\mathcal{A}$ and $\mathcal{B}$.
\end{thm}

It is then not hard to see that for $\A$ and $\B$ countable, we have that $\A\equiv_{\infty\om}\B$ if and only if $\A\isom\B$.
Since Karp's characterization shows that $\equiv_{\infty\om}$ is absolute, we get that, for uncountable structures, $\A\equiv_{\infty\om}\B$ iff they are isomorphic when they are made countable:

\begin{lemma}[Essentially Barwise \cite{Bar73}]\label{potential isomorphism}
The following are equivalent:
\begin{enumerate}
\item $\A\equiv_{\infty\om}\B$,
\item on every generic extension where $\A$ and $\B$ are countable, $\A\isom \B$,
\item on some generic extension where $\A$ and $\B$ are countable, $\A\isom \B$.
\end{enumerate}
\end{lemma}



\bigskip

As an immediate corollary, we have the following.

\begin{cor} $\A\equiv_{\infty\omega}\B$ implies $\A\equivgen\B$.
\end{cor}

This lets us connect $\equivin$-equivalence  and generic Muchnik reducibility in a strong way:

\begin{lemma}\label{copying}
Let $\A$ be a structure.
The following are equivalent:
\begin{enumerate}
\item $\A\legen \B$ for some countable structure $\B$.
\item $\A\equivgen \B$ for some countable structure $\B$.
\item $\A\equiv_{\infty\omega}\B$ for some countable structure $\B$.
\end{enumerate}
\end{lemma}
\begin{proof} Clearly $(3)$ implies $(2)$ and $(2)$ implies $(1)$.  To see that $(1)$ implies $(3)$, suppose $\A\legen\B$ for $\B$ countable, let $B$ be a copy of $\B$ in $V$, and let $V[G]$ be a generic extension in which $\A$ is countable. Then in $V[G]$, there is some index $e$ such that, for the $e$th Turing machine $\Phi_e$,
$\Phi_e^B\cong\A$.  This means that in $V$, $\Phi_e^B$ must be total, and so $\Phi_e^B$ is a copy of $\mathcal{A}$ which lives in $V$.
\end{proof}

\subsection{Examples}
\label{Examples} 

We present below some examples of uncountable structures whose complexity in terms of $\legen$ we have been able to analyze.

\begin{exmp}\label{powerset example}  
Let $\mathcal{U}$ be the structure with domain $\omega\sqcup\P (\omega)$, with signature consisting of only the $\in$-relation on $\omega\times\P (\omega)$. 
\end{exmp}

\begin{prop}
$\mathcal{U}\equivgen 0$, where $0$ is the empty structure.  
\end{prop}

\begin{proof}
We will show there is a computable structure $\mathcal{S}$ that is $\equivin$-equivalent to $\mathcal{U}$.
By the absoluteness of $\equivin$, we will then have that in any generic extension that makes $\mathcal{U}$ countable, $\mathcal{S}$ and $\mathcal{U}$ are still $\equivin$-equivalent, and, hence, isomorphic. 

We note that the orbit in $\mathcal{U}$ of a tuple of sets $\bar{X}$ is determined by the cardinalities of the Boolean combinations of the sets $X_i$.  To guarantee that we have a back-and-forth family of finite partial isomorphisms, we let $\mathcal{S}$ consist of $\omega$ together with a family of sets $P$ having the following properties:

\begin{itemize}
\item  $P$ is an algebra of sets; i.e., it is closed under union, intersection, and complement,
\item  $P$ includes all finite sets, 
\item  if $X\in P$ is infinite, then there are disjoint $Y,Z\in P$, both infinite, such that $Y\cup Z = X$.
\end{itemize}

We can easily find such an $\mathcal{S}$ which is computable.    
We could, for example, take the family of primitive recursive sets.  
\end{proof}    

Similarly, the field of complex numbers is essentially computable.

\begin{exmp}

Let $\mathcal{C}=(\mathbb{C}; +, \times)$.  This is $\equivin$-equivalent to the algebraically closed field of countably infinite transcendence degree and characteristic zero.  By a well-known result of Rabin \cite{Rab60}, this has a computable copy.  Then $\mathcal{C}$ has minimal complexity; that is, $\mathcal{C}$ has a computable copy in every generic extension in which it is countable.

\end{exmp}

If we consider a variant of $\mathcal{U}$ in which the elements of $\omega$ have names, we reach the opposite end of the complexity spectrum:

\begin{exmp}\label{power set with successor}

Let $\mathcal{W}$ be the expansion of $\mathcal{U}$ including the successor relation on $\omega$.  Then any $\omega$-copy (\ref{remark on copies}) of $\mathcal{W}$ computes every real in the ground model $V$, so given any countable $\mathcal{A}\in V$ we have $\mathcal{A}\legen\W$.

\end{exmp}

The situation is the same with respect to the real numbers.

\begin{exmp} The field of real numbers $\mathcal{R}=(\RR; +, \times)$ is, like $\W$, maximally complicated with respect to countable structures: for every countable structure $\mathcal{A}$, we have $\mathcal{A}\le_w^*\mathcal{R}$. To see this, suppose $V[G]$ is a generic extension in which $\mathcal{R}$ has an $\omega$-copy, $\R$. First, note that the standard ordering $<_\RR$ is both $\Sigma_1$ and $\Pi_1$ over $\mathcal{R}$, and so the corresponding relation on $\R$ is computable relative to $\R$.\footnote{That is, $<_\RR$ is a {\em relatively intrinsically computable} relation.}

Now fix a real in the ground model $b\in\mathcal{R}$ and let $\hat{b}\in\R$ be the corresponding element of the $\omega$-copy. Since $<_\R$ is computable from the atomic diagram of $\R$, so is the cut corresponding to $b$, so every real in the ground model is computable from the atomic diagram of $\R$. Since $\R$ was an arbitrary $\omega$-copy of $\mathcal{R}$ in an arbitrary generic extension, it follows that $\mathcal{R}\ge_w^*\mathcal{A}$ for every countable $\mathcal{A}\in V$.
\end{exmp}

We would now like to compare the structures $\mathcal{R}$ and $\mathcal{W}$ under $\leq^*_w$.  It is east to show the following.      

\begin{prop}

$\mathcal{R}\geq^*_w\mathcal{W}$

\end{prop}

\begin{proof}

We can use the elements of $\mathcal{R}$ in the interval $[0,1)$ to enumerate the subsets of $\omega$ in $V$.  To each real $r$ in the interval, we associate the set $A_r$ consisting of those $n$ such that the $n^{th}$ term in the binary expansion of $r$ is $1$.
Minimal care has to be taken for double binary representations: if we assume no binary expansion ends up in an infinite string of $1$s, we then need to add those sets.
\end{proof}


\begin{question}

Do we have $\mathcal{W}\geq^*_w\mathcal{R}$?

\end{question}

For an elementary extension $\mathcal{M}$ of $\mathcal{R}$ that is $\omega$-saturated, we have $\mathcal{W}\geq_w^*\mathcal{M}$.  More generally, we have the following.  

\begin{prop}

Let $\mathcal{M}$ be an $\omega$-saturated model of a decidable complete elementary first order theory $T$.  Then $\mathcal{W}\geq_w^*\mathcal{M}$.  

\end{prop}

\begin{proof}

Macintyre and Marker \cite{MM84} showed that for an enumeration $\mathcal{E}$ of a Scott set $\mathcal{S}$, and an elementary first order theory $T$ in $\mathcal{S}$, 
$\mathcal{E}$ computes the complete diagram of a recursively saturated model of $T$ realizing exactly the types in $\mathcal{S}$ that are consistent with $T$.  After we collapse the cardinal so that $\mathcal{W}$ becomes computable, it computes an enumeration $\mathcal{E}$ of the Scott set $\mathcal{S}$ consisting of the subsets of $\omega$ in $\mathcal{W}$.  Now, the theory of $\mathcal{M}$ is in $\mathcal{S}$, and the types realized in $\mathcal{M}$ are exactly those in $\mathcal{S}$ that are consistent with $T$.  Then the result of 
Macintyre and Marker yields a recursively saturated model realizing exactly these types.  This model is isomorphic to the collapse of $\mathcal{M}$.         
\end{proof}

Finally, uncountable well-orderings live strictly between the two extremes.

\begin{exmp}
\label{complexity of om1} 

The linear order $\om_1=(\omega_1, <)$ computes precisely those countable structures which are Muchnik reducible to some countable well-ordering. One direction is obvious; in the other direction, suppose $\mathcal{A}\le_w^*\om_1$ is countable, and let $V[G]$ be a forcing extension in which $\om_1$ is countable. Then $V[G]$ satisfies ``$\mathcal{A}$ is weakly reducible to a countable well-ordering," which is $\Sigma^1_2$ via \ref{obsv}, and so already true in $V$ by Shoenfield absoluteness.
\end{exmp}

\begin{proposition} 

$\mathcal{R}>_w^*\om_1$ strictly. 

\end{proposition}

\begin{proof} 

To see that $\mathcal{R}\not\le_w^* \om_1$, fix some non-computable real $r\in\mathcal{R}$. Then the cut corresponding to $r$, and hence $r$ itself, is computable in any $\omega$-copy $\R$ of $\mathcal{R}$ in any generic extension since the ordering relation is both $\Sigma_1$ and $\Pi_1$.  On the other hand, by a result of Richter \cite{Ric81}, the only sets computable in all copies of a countable linear ordering are the computable sets, so in any generic extension in which $\omega_1$ is countable there will be $\omega$-copies of $\om_1$ whose atomic diagrams do not compute $r$.  

To see that $\om_1\le_w^*\mathcal{R}$, suppose $V[G]$ is a generic extension in which $\mathcal{R}$ is countable, and let $\R\in V[G]$ be a copy of $\mathcal{R}$ with domain $\om$.  Now $\R$ computes an enumeration of the sets coded by the cuts in $\mathcal{R}$---the reals in $V$.  Some of the reals code linear orderings.  For an ordering $r$ coded in $\R$, if $r$ is not a well ordering, this is witnessed by a decreasing sequence $d$, also coded in $\R$.  A countable well ordering in $V$ is isomorphic to a countable ordinal, so it stays well ordered in $V[G]$.  Using $\R\dprime$, we get an $\omega$-sequence of well-orderings:  For $a\in \R$, we take the ordering coded by $a$, if this is a well ordering, and otherwise, we have a finite ordering.  The result is an ordering of type $\omega_1^V$.  
Now, we apply in $V[G]$ the theorem saying that, for any set $X$ and any linear order $\mathcal{L}$, if $X\dprime$ computes a copy of $\L$ then $X$ computes a copy of $\om\cdot\L$ (\cite{AK00}, Theorem 9.11).  Since $\omega_1^V\isom \om\cdot\omega_1^V$, our $\R$ computes a copy of $\omega_1^V$.  
\end{proof}

By a similar argument, we also have $\mathcal{W}>_w^*\om_1$.

\subsection{Generic presentability}
\label{genericpresentability} 

We end this section with some basic results about generic presentability, which will be used in the remaining sections.
First, we recall Solovay's proof that every generically presentable \it set \rm is already present in the ground model.

\begin{thm}[Solovay \cite{Sol70}]\label{Solovay}  Suppose $S$ is a set present in every generic extension of $V$ by $\PP $. 
Then $S\in V$.
\end{thm}  

\begin{proof} 
Let us start by proving that there is a single name $\nu$ such that $\nu[G]=S$ for every $\PP$-generic $G$.
Fix a generic $H$ so that we can talk about $S$ within $V[H]$.
If there is no single name $\nu$ such that $\nu[G]=S$ for every $\PP$-generic $G$, then for every name $\nu\in V^\PP $ there is some generic $G$ over $V[H]$ with $\nu[G]\not=S$.
Then, for each $\nu\in V^\PP $, the set $\{p\in \PP : p\Vdash\nu\not=S\}$ is dense and is in $V[H]$ since $S\in V[H]$.
But then, if $G$ is $\PP$-generic over $V[H]$,  $S\not\in V[G]$, so $S$ is not generically presentable.

Let us now go back to the proof of the theorem.
Suppose $S\not\in V$.  
Let $S$ be a counterexample of minimal rank $\alpha$.  
Then each element of $S$ is in $V$, and in particular in the set $V_\alpha$. Let 
\[\hat{S}=\{x\in V_\alpha: \exists p\in \PP (p\Vdash x\in\nu)\}.\] 
Then clearly $\hat{S}=S$, and so $S\in V$.
\end{proof}

The goal of this paper is to investigate the extent to which Solovay's theorem holds for structures, and, in particular, to understand how it interacts with computability in generic extensions. 
As we will see in Section \ref{om2bad}, the proof of Solovay's theorem cannot be naively extended to structures. However, one key step still holds, and this greatly simplifies arguments involving generic presentability.  

\begin{lem}
\label{Name} 

Suppose $\mathcal{N}$ is generically presentable by $\PP $. Then there is some name $\nu\in V^\PP $ such that 
$\nu[G]\cong\mathcal{N}$ for every $\PP $-generic $G$.
\end{lem}

If $\PP , \mathcal{N}$, and $\nu$ are as above, we say that $\mathcal{N}$ is {\em generically presentable by $\PP $ via $\nu$}. 

\begin{proof} This is a generalization of the first paragraph of the proof of Solovay's Theorem (Theorem \ref{Solovay}).  Fix $H_0$ $\PP $-generic over $V$. Since $\mathcal{N}$ is generically presentable, there is a name $\mu$ such that $\mu[H_0]\cong \mathcal{N}$. Looking at the product forcing $\PP ^2$, there exist a condition $p\in H_0$ and a name $\nu$ such that \[(p, 1)\Vdash_{\PP ^2}\mu[\gl]\cong\nu[\gr].\]

We claim that $\nu$ names a copy of $\mathcal{N}$ in all generic extensions. To see this, fix a $\PP $-generic filter $G_1$.  We can find $G_0, H_1$ such that 
\begin{itemize}
\item $p\subseteq G_0$ and $G_0$ is $\PP $-generic over $V[G_1]$,
\item $H_1$ is $\PP $-generic over $V[H_0]$, and
\item $G_0\times H_1$ is $\PP ^2$-generic over $V$. 
\end{itemize}
Then we have 
\[
\nu[G_1]\cong\mu[G_0]\cong\nu[H_1]\cong\mu[H_0]\cong\mathcal{N}.
\qedhere
\] 
\end{proof}

Finally, although it is not directly useful to the results of this paper, we note that being ``generically generically presentable" is the same as being generically presentable.  

\begin{prop} If $\mathcal{A}$ is generically presentable over $V[G]$ for every $\PP $-generic filter $G$, then $\mathcal{A}$ is generically presentable over $V$.
\end{prop}

\begin{proof} 
The proof is identical to the proof of Lemma \ref{Name} above, except that instead of a name $\nu$ for $\A$, we work with a name $\hat{\nu}$ for a pair $\langle \mathbb{Q}, \mu\rangle$ where $\mathbb{Q}$ is a poset and $\mu$ is a $\mathbb{Q}$-name; this is just the ``two-step iteration" version of Lemma \ref{Name}.
\end{proof}

\section{Generic presentability and $\omega_2$}

In this section and the next, we address the question ``when do  generically presentable structures have copies in $V$?" This section focuses on a forcing-theoretic aspect of the question.   For which forcing notions $\PP $ do we have copies in $V$ for all structures generically presentable by $\PP $ with universe $\omega$?  Surprisingly, this is entirely determined by how $\PP $ affects cardinals: $\om_2$ remains uncountable after forcing with $\PP $ if and only if every structure generically presentable by $\PP $ on $\omega$ has a copy in $V$.

As a consequence of proving the left-to-right direction of this result, we also give a new proof of the result due to Harrington that counterexamples to Vaught's conjecture must have models of arbitrarily high Scott rank in $\om_2$. The right-to-left direction follows from a construction of Laskowski and Shelah \cite{SL93}.  

\subsection{Scott Analysis.}     

We begin by reviewing the {\em Scott analysis} of a structure.  Scott \cite{Sco65} proved that for every countable structure $\A$, there is an infinitary sentence $\sigma$ of $L_{\omega_1\omega}$ such that the countable models of $\sigma$ are exactly the isomorphic copies of $\A$.  Such a sentence is called a \emph{Scott sentence}. 

There are several definitions of \emph{Scott rank} in the literature (see, in particular, \cite{Bar75}, \cite{AK00}, \cite{SM}, \cite{CKM}, and \cite{MonSR}).  The definitions give slightly different values.  However, all of the definitions assign countable Scott ranks to countable structures.  In general, the complexity of the Scott sentence is only a little greater than the Scott rank of the structure.  If one definition assigns a computable ordinal Scott rank, then the other definitions do as well, and then there is a Scott sentence that is $\Sigma_\alpha$, for some computable ordinal $\alpha$.  The definition that we give below is the one used by Sacks \cite{Sac07}.  We begin by defining a family of definable expansions of~$\A$.       

\begin{definition}
\label{Scott: predicate iteration}

We first have to review the Scott analysis of a structure. 
(See \cite{Sac07}.)
For each $\a$, we define a fragment $\Lan^\A_\a$ of $L_{\infty,\om}$ as follows:
\begin{itemize}
\item Let $\Lan^\A_0$ consists of the first order formulas.
\item Given $\Lan^\A_\a$, for each complete non-principal type $\Phi(x)\subseteq \Lan^\A_\a$ realized in $\A$, add the formula $\bigwedge \Phi(x)$ to $\Lan^\A_{\a+1}$, and close under first order connectives and quantifiers.
\item At limit levels, take unions. 
\end{itemize}
For each $\a$ there is a natural way to expand $\mathcal{A}$ to a $\Lan^\A_\a$-structure $\mathcal{A}_\a$; we will abuse notation by omitting the subscript, since no confusion will arise.
\end{definition} 

At some step $\a$, $\A$ becomes $\Lan^\A_\a$-atomic, in the sense that all $\Lan^\A_\a$-types are principal.

\begin{definition} \label{def:Scott rank}
The \emph{Scott rank} of $\A$, $sr(\A)$, is the least ordinal $\a$ such that $\A$ is an $\Lan^\A_\a$-atomic structure.
\end{definition} 

\begin{lemma}\label{presenting Scott analysis}
If $\A$ is generically presentable, then, for every ordinal $\b$, $\Lan^\A_\b\in V$.
\end{lemma}
\begin{proof}
First, let us remark that we can code the formulas in $\Lan^\A_\b$ by sets:
for instance, we code an infinitary conjunction of formulas $\psi_i$ by a pair, the first element being a code that means ``conjunction'' and the second element being the set of codes for the formulas $\psi_i$ --- say, by defining $code(\bigwedge_{i\in I} \psi_i(x)) = \la 17, \{code(\psi_i(x)):i\in I\}\ra$. This is quite standard, so we let the reader fill in the details.

The one important detail is that we are not coding infinitary conjunctions using sequences of formulas, but using sets where the order of the formulas does not matter. 
The key point is that if we have different presentations of a structure $\A$, the types realized in each presentation are the same as sets.
We can then prove by induction on $\b$, that $\Lan^A_\b$ is a set that is independent of the presentation of $\A$.
Since $\A$ is generically presentable, say by a forcing notion $\PP$, the language $\Lan^A_\b$ belongs to all $\PP$-forcing extensions of $V$, and so by Solovay's Theorem \ref{Solovay}, we get that $\Lan^A_\b$ belongs to $V$.
\end{proof}

\begin{definition}\label{Ahat}
Given a structure $\A$, let $\hat{\mathcal{L}}$ be the language containing a relation symbol for each formula in $\Lan^\A_{sr(\A)}$ (the {\em Morleyization} of $\Lan^\A_{sr(\A)}$), and let $\Ahat$ be the natural expansion of $\A$ to the language $\Lhat$. Note that if $\A$ is generically presentable, then $\Lhat\in V$ since $\Lan^\A_{sr(\A)}\in V$.
\end{definition}

Notice that $\Ahat$ is atomic in a very strong way: each $\Lhat$-type is generated by a quantifier-free $\Lhat$-formula.


\begin{remark} Throughout this section we will tacitly assume that $\mathcal{L}$ (and hence $\Lhat$ as well) is no larger than $\A$; that is, that the statement ``$\vert\mathcal{L}\vert\le\vert\A\vert$" is true in every forcing extension by $\mathbb{P}$ (where $\mathbb{P}$ is a forcing generically presenting $\A$). This assumption is used, for example, in \ref{scottom1} below, and is necessary for straightforwardly applying the facts about amalgamation we will prove in section \ref{Ages}. Note that this assumption holds for the vast majority of natural structures.
\end{remark}


\begin{lemma}\label{le: Ahat}
If $\A$ is generically presentable, then so is $\Ahat$.
\end{lemma}
\begin{proof}
We already showed that $\Lan^\A_{sr(\A)}\in V$, so $\Lhat\in V$.
There is only one way to expand $\A$ to the $\Lhat$-structure $\Ahat$.
So, $\Ahat$ has a presentation with domain $\om$ in every generic extension of $V$ where $\A$ does.
\end{proof}

\begin{prop}\label{scottom1} 
Suppose $\A$ is generically presentable by a forcing notion that does not collapse $\om_1$.  Then $\A$ has a copy in $V$ with domain $\om$.
\end{prop}   
\begin{proof}
Let $\PP$ be a forcing notion that does not collapse $\om_1$, and for which $\A$ is generically presentable.
Since $\Ahat$ is generically presentable, and $\Lhat\in V$, we have that $Th_\Lhat(\Ahat)$, the $\Lhat$-theory of $\Ahat$, is a set of $\Lhat$ sentences that belongs to all $\PP$-generic exensions.  Thus, $Th_\Lhat(\Ahat)\in V$.

In all of these extensions, $\Lhat$ is countable (because $\A$ is), and, hence, $\Lhat$ cannot be uncountable in $V$.
Otherwise, there would be an injection from $\om_1$ into $\Lhat$, and since $\PP$ does not collapse $\om_1$, $\Lhat$ would stay uncountable in $V[G]$.

Now, in  each of these generic extensions, $\Ahat$ is the unique countable atomic model of $Th_\Lhat(\Ahat)$.
The existence of such a model is a $\Si^1_1$ statement with $Th_\Lhat(\Ahat)$ as parameter. 
By absoluteness, this must be true in $V$ too, and by the uniqueness of $\Ahat$ in $V[G]$, this model must be isomorphic to $\Ahat$.
\end{proof}

\subsection{Keeping $\om_2$ uncountable.}\label{Ages}

We now turn to the {\em Fra\"iss\'e limit} construction, first used in \cite{Fra00}:

\begin{definition} Fix a relational language $\mathcal{L}$. For an $\mathcal{L}$-structure $\B$, we denote by $\bfK_\B$ the set of finite substructures of $\B$, and we call $\bfK_\B$ the {\em age} of $\B$. For $\bfK$ a set of finite structures and $\A$ a structure, we say that $\A$ is the {\em Fra\"iss\'e limit} of $\bfK$ if $\bfK_\A=\bfK$ and the set of isomorphisms between finite substructures of $\A$ has the back-and-forth property.
\end{definition}

It is clear from the definition that if $\A$ and $\B$ are countable Fra\"iss\'e limits for the same age $\bfK$, then $\A\cong\B$.  A given age may have non-isomorphic uncountable Fra\"iss\'e limits.  For example, if $\bfK$ is the set of finite linear orderings, the Fra\"iss\'e limits are the dense linear orderings without endpoints, and there are many --- in fact, $2^{\aleph_1}$  many, the most possible --- non-isomorphic ones of cardinality $\aleph_1$.      

\begin{lemma}\label{presenting ages} If $\A$ is generically presentable, then $\bfK_\A\in V$.
\end{lemma}

\begin{proof} 
This follows from Solovay's Theorem \ref{Solovay}: $\bfK_\A$ is a set of finite structures that is independent of the presentation of $\A$.
\end{proof}

Using the same argument as in Proposition \ref{scottom1}, we get a bound on the size of $\hat{\mathcal{L}}$ and $\bfK_{\Ahat}$:

\begin{corollary}\label{anothersizelemma} If $\A$ is generically presentable by a forcing not making $\om_{2}$ countable, then $\hat{\mathcal{L}}$ and $\bfK_{\Ahat}$ have size $\leq\aleph_1$ in $V$.
\end{corollary}

Fra\"iss\'e \cite{Fra00} proved that if $\bfK$ is a countable set of finite structures satisfying the Hereditary Property ($HP$), the Joint Embedding Property ($JEP$) and the Amalgamation Property ($AP$), then it has a Fra\"iss\'e limit (see 6.1 of \cite{shorter} for definitions). The next lemma says that this is still the case when $\bfK$ has size $\aleph_1$.  The earliest reference we know is Delhomme, Pouzet, Sagi, and Sauer \cite[Corollary 2, p.\ 1378]{DPSS}.  We give the proof because we want to make clear that the result does not automatically generalize to ages of size $>\aleph_1$; and indeed, we will see in the next subsection that there is an age of size $\aleph_2$ with no limit (\ref{nobiglimit}).    

\begin{lemma}
\label{lem: Fraisse aleph 1}
Let $\bfK$ be a family of $\aleph_1$ finite structures on a relational language $\L$ of size $\leq \aleph_1$. If $\bfK$ has HP, JEP, and AP, then there is a Fra\"iss\'e limit $\A$ with age $\bfK$.
\end{lemma}

\begin{proof}
The key is the following:
\begin{quote}
{\em Claim:} Suppose we have embeddings $\A\to\B$ and $\A\to \C$ where $\A,\B\in \bfK$ and $\C$ is countable and its age is a subset of $\bfK$. Then there is a countable structure $\D$, whose age is a subset of $\bfK$, and which amalgamates these embeddings.
\end{quote}
To prove the claim, write $\C$ as the union of an increasing sequence $\{\C_n:n\in\om\}$ where each $\C_n\in\bfK$, and with $\C_0=\A$.
Let $\D_0=\B$, and note that we have an embedding from $\C_0$ to $\D_0$. Given $\D_n$, by induction we will have an embedding from $\C_n$ into $\D_n$, and $\D_n$ will be an element of $\bfK$; and by definition we have an embedding from $\C_n$ into $\C_{n+1}$. We then form $\D_{n+1}$ by amalgamating the embeddings $\C_n\rightarrow\C_{n+1}$ and $\C_n\rightarrow \D_n$ within $\bfK$. The direct limit $\D$ of the $\D_i$ is the desired amalgamation.

Now we prove the lemma. Suppose $\bfK$ is such a family of finite structures. There is a sequence $(\A_\xi)_{\xi\in\om_1}$ of structures such that:
\begin{itemize}
\item $\xi_0<\xi_1\implies \A_{\xi_0}\subseteq\A_{\xi_1}$;
\item each $\A_\xi$ is countable and its age is a subset of $\bfK$; and
\item for every $\xi\in\om_1$ and $\B,\C\in\KK$ and every pair of embeddings $\B\to\C$ and $\B\to\A_\xi$, there is $\g>\xi$ and an embedding $\C\to\A_\g$ compatible with the inclusion $\A_\xi\to\A_\g$.
\end{itemize} The union $\A$ of the $\A_\xi$ clearly has age $\bfK$.  It is clear from the construction that the set of finite partial isomorphisms has the back-and-forth property.  
\end{proof}

Note that the limit $\mathcal{A}$ constructed above need not be $\aleph_1$-homogeneous or unique.

\begin{corollary}\label{presenting ultrahomogeneous}  
Let $\B$ be an $\L$-structure that lives in an extension of the universe and is $\omega$-homogeneous in the sense that the family of isomorphisms between finite substructures has the back-and-forth property.  
Suppose $\B$ is generically presentable, and  $\vert\bfK_\B\vert,\vert\L\vert\leq \aleph_1$ in $V$. Then in $V$ there is a structure $\infty\omega$-equivalent to $\B$.
\end{corollary}

\begin{proof}
Since $\B$ is generically presentable, we have that $\bfK_\B\in V$ by Lemma \ref{presenting ages}. Since $\B$ is $\omega$-homogeneous, $\bfK_\B$ has $HP$, $JEP$ and $AP$ in any model where $\B$ lives; in particular, $\bfK_\B$ has these properties in $V$. Since $\vert\bfK_\B\vert\leq \aleph_1$ and $\vert\L\vert\leq\aleph_1$ in $V$, by Lemma \ref{lem: Fraisse aleph 1} we have that $\bfK_\B$ has a Fra\"iss\'e limit $\F$ in $V$. In a generic extension presenting $\B$, the age $\bfK_\B$---and, hence, the Fra\"iss\'e limit $\F$---will be countable.  Then $\F\cong\B$, by the uniqueness of countable Fra\"{i}ss\'{e} limits, so $\F$ is the required structure $\infty\omega$-equivalent to $\mathcal{B}$ which lives in $V$.  
\end{proof}

We are now ready to prove the main positive result of this section.

\begin{theorem}
\label{colom1}
Let $\A$ be a structure in an extension $W$ of $V$, and assume $\A$ is countable in $W$.
Suppose $\A$ is generically presentable by a forcing notion that does not make $\om_2$ countable.  
Then there is a copy of $\A$ in $V$, with cardinality at most $\aleph_1$ in~$V$.
\end{theorem} 

(Recall that the conclusion of the theorem should be read as: there is a structure $\B\in V$ of cardinality at most $\aleph_1$ in~$V$, which is isomorphic to $\A$ in $W$. 
See Remark \ref{remark on copies}.)

\begin{proof} 

Let $\mathcal{L}$ be the language of $\A$. Since $\A$ is generically presentable, by Lemma \ref{presenting Scott analysis}, we know that $\Lhat$ is in $V$ and $\Ahat$ is generically presentable.  Consider some generic extension $V[G]$ by a forcing which generically presents $\A$ and which does not make $\om_2$ countable. Using in $V[G]$ the fact that Scott ranks of countable structures are countable, since $\om_2^V$ is still uncountable in $V[G]$ the language $\hat{\L}$ has size $\leq\aleph_1$ in $V$.  This implies that $\bfK_{\Ahat}$ is in $V$ by Lemmas \ref{le: Ahat} and \ref{presenting ages} and has size $\leq\aleph_1$ in $V$ by \ref{anothersizelemma}. Now, we can apply Corollary \ref{presenting ultrahomogeneous} to get a copy of $\Ahat$ which lives in $V$ (of course, $\Ahat$ need not be countable in $V$). Intuitively, we now want to take the reduct of this copy to $\mathcal{L}$, but $\hat{\mathcal{L}}$ need not include $\mathcal{L}$ (for instance, if two $\mathcal{L}$-symbols have the same interpretation); instead, from $\Ahat$ we can now ``decode" the correct interpretations of each of the symbols in $\mathcal{L}$, and thus produce a copy of $\A$ itself.
\end{proof}

Note that Theorem \ref{colom1} does not directly imply Remark \ref{scottom1}, since the latter concludes that the generically presentable structure in question has a {\em countable} copy in $V$. 

\bigskip

We may apply Theorem \ref{colom1} to prove the following.

\begin{theorem}[Harrington, unpublished]
If $T$ is a counterexample to Vaught's Conjecture, then for each $\alpha < \omega_2$, $T$ has a model of size $\aleph_1$ with Scott rank $\geq\alpha$.
\end{theorem}
\begin{proof} 
Recall that if $T$ is a counterexample to Vaught's conjecture it has countable models of arbitrary Scott rank below $\om_1$.
Being a counterexample to Vaught's conjecture is a $\Pi^1_2$ property (\cite{Mor70}; see also \cite{Sac07}, Proposition 5.1) and hence absolute.

Let $\PP=\om_1^{<\om}$ be the usual Levy collapse of $\om_1$ and let $G$ be $\PP$-generic. Since $T$ is a counterexample to Vaught's conjecture, in $V[G]$ we have a countable model $\B$ of Scott rank $\geq\alpha$. We claim that $\B$ is generically presentable over $V$ by $\PP$.
This would give us the claimed result: since $\PP $ does not collapse $\om_2$, by Theorem  \ref {colom1}, we would have a copy of $\B$ of size $\aleph_1$ in $V$, and since Scott rank is absolute, this copy is as wanted.

Let us now prove that $\B$ is indeed generically presentable over $V$ by $\PP$.
Fix a $\PP$-generic $G$, and a name $\nu\in V$ for $\B$ in $V[G]$.
The idea is to show that whenever we have two mutually generics, $\nu$ gives us non-isomorphic models of $T$, and hence if we add continuum many mutually generics, we get continuum many different countable models of $T$. 
It is enough to show that for every $H$ which is $\PP$-generic over $V[G]$, there is a structure in $V[H]$ that is isomorphic within $V[G][H]$ to $\nu[G]$.
Suppose this is not the case.
Then, there is some $\PP$ condition $q$ forcing, in $V[G]$, that there is no copy of $\B$ in $V[H]$ for every extension $H$ of $q$.
Since this sentence is invariant under finite changes in $H$, we can assume $q$ is the empty condition.
Now, since $G$ is generic over $V$, there is a $p\subset G$ which forces this.
Assume that $p$ also forces that $\nu[G]$ is a model of $T$.
So, we get that, the $\PP\times\PP$-condition $(p,p)$ forces that $\nu[G]$ and $\nu[H]$ are not isomorphic in $V[G][H]$, and are both countable models of $T$.

Now, consider a forcing notion $\QQ$ that adds perfectly many $\PP$-generics.
(This is quite standard: for instance let $\QQ$ be the set of finite partial maps from $2^{<\om}$ to $\om_1^{<\om}$ and then obtain the $\PP$ generics by concatenating the $\om_1^{<\om}$-strings along each path in $2^\om$.)
In this generic extension of $V$, we get continuum many countable models of $T$. Since being a counterexample to Vaught's conjecture is absolute, this is a contradiction.
\end{proof}

\begin{remark} Recently, Baldwin, S.-D. Friedman, Koerwien, and Laskowski \cite{BFKL14} have given a new proof of Harrington's result using similar genericity arguments; their proof uses a generic version of the Morley tree, which they show is invariant across forcing extensions.
\end{remark}

Finally, we can use Theorem \ref{colom1} to give a partial positive answer to Slaman's question:

\begin{corollary}\label{computinggen} Suppose $\A$ lives in an extension of the universe and $\A\legen\B$ for some $\B\in V$ with cardinality $\le\aleph_1$. 
Then $\A$ has a copy in $V$.
\end{corollary}

\begin{proof} Let $\PP $ be a forcing notion that collapses $\om_1$ while keeping $\om_2$ uncountable, such as $\PP =\omega_1^{<\omega}$. Let $V[G]$ be a generic extension by $\PP $. Then $\B$ is countable in $V[G]$, and, a fortiori, there is a copy of $\A$ in $V[G]$.  
It follows that $\A$ is $\PP$-generically presentable. Then by Theorem~$\ref{colom1}$, there is a copy of $\A$ in $V$.
\end{proof}

\subsection{Collapsing $\om_2$ to $\om$.}\label{om2bad}

We close this section by presenting a strong negative result, coming from a construction due to Shelah and Laskowski \cite{SL93}. Throughout the rest of this section, we abbreviate the linear order $(\omega_2, <)$ by ``$\omega_2$." 

\begin{theorem}\label{colom2} There is a structure $\mathcal{A}$, generically presentable by any forcing making $\omega_2$ countable, but with no copy in $V$.
\end{theorem}

\begin{proof} 
Laskowski and Shelah \cite{SL93} gave an example of an elementary first order theory $T$, in a countable language, such that:
\begin{enumerate}
\item The language has a sort $\Vtt$ such that, for every model $\M$ of $T$ and every subset $A\subseteq \Vtt^\M$, $T(A)$ has an atomic model if and only if $|A|\leq \aleph_1$.
\item $T$ has a countable model $\M_0$ such that $\Vtt^{\M_0}$ is totally indiscernible in the sense that any permutation of $\Vtt^{\M_0}$ extends to an automorphism of $\M_0$. 
Furthermore, $\M_0$ is atomic over $\Vtt^{\M_0}$.
\end{enumerate}

For $\mathcal{C}$ a countable structure, let $\mathcal{M}_\mathcal{C}$ be the two-sorted structure with one sort corresponding to a copy of $\mathcal{C}$, one sort corresponding to a copy of $\mathcal{M}_0$, and with a function symbol $f$ providing a bijection between $\mathcal{C}$ and $\Vtt^\mathcal{M}_0$. 
Since the elements of $\Vtt^{\mathcal{M}_0}$ are totally indiscernible, any two choices of $f$ yield isomorphic structures, so $\mathcal{M}_\mathcal{C}$ is well-defined.

Now consider the ``structure'' $\mathcal{M}_{\om_2}$ which lives in any extension of the universe where $\om_2$ is countable.  
Thus, $\M_{\om_2}$ is generically presentable by $Col(\om_2,\om)$.
However, there is no copy of $\M_{\om_2}$ in $V$: 
Since if the first sort is really $\om_2$, of size $\aleph_2$, then in the second sort, the predicate $\Vtt$ has size $\aleph_2$.
But, by the assumption on $\M_0$, $\M_\C$ is always atomic over $\C$ (a fact that is absolute), and by the assumption on $T$, $T(\Vtt^{\M_{\om_2}})$ has no atomic models.
\end{proof}    

The structure of Laskowski and Shelah also provides a counterexample to a natural extension of Lemma \ref{lem: Fraisse aleph 1}.  

\begin{cor}\label{nobiglimit}
There is an age $S$ of size $\aleph_2$ with the Hereditary, Joint Embedding, and Amalgamation properties but for which there is no Fra\"iss\'e limit.
\end{cor}

\begin{proof} 

Consider the theory $T(A) = Th(\mathcal{M}_0,a_{a\in A})$, where $A = A^\mathcal{M}$ has size $\aleph_2$.  The principal types are dense, but $T(A)$ has no atomic model.  We add predicate symbols for the principal types.  For $B\subseteq A$ of size up to $\aleph_1$, there is an atomic model of the corresponding theory $T(B) = Th(\mathcal{M}_0,a_{a\in B})$.  Let $K$ consist of the finite substructures of the atomic models of the theories $T(B)$.  In total, what we have is appropriate to be the age for an atomic model of $T(A)$.  That is, we have the Hereditary, Joint Embedding, and Amalgamation properties  (essentially \cite{SL93}, pg. 3).  However, any Fra\"iss\'e limit of $S$ would yield an atomic model of $T(A)$, so the Fra\"iss\'e limit cannot exist.
\end{proof}

\section{Generically presentable rigid structures}\label{presentrigid}

In the previous section, we gave a complete characterization of those posets $\mathbb{P}$ with the property that every structure generically presentable by $\mathbb{P}$ has a copy already in the ground model. In this section, we examine the dual question: what properties of {\em structures} ensure that generic presentability implies the existence of a copy in the ground model? Specifically, we extend Solovay's Theorem \ref{Solovay} to structures that are sufficiently ``set-like:"

\begin{thm}\label{Rig} Suppose $\mathcal{N}$ is a rigid structure with language $\L$, generically presentable by $\mathbb{P}$. Then there is a copy of $\mathcal{N}$ in $V$.
\end{thm}

\begin{proof} We assume $\mathcal{L}$ is relational; note that since $\mathcal{N}$ is generically presentable, and $\mathcal{L}$ is a set, we have $\mathcal{L}\in V$. By Lemma \ref{Name}, there is a single name $\nu$ which names a copy of $\mathcal{N}$ in any generic extension by $\mathbb{P}$. Without loss of generality, we may assume that $\Vdash dom(\nu)=\omega$. On $\omega\times\mathbb{P}$, we then define the relation $\equiv$ as follows: \[(a, p)\equiv (b, q)\iff (p, q)\Vdash_{\mathbb{P}^2}``\{(a, b)\}\text{ extends to an isomorphism $\nu[\gl]\cong \nu[\gr]$}."\] If $(a, p)\equiv (a, p)$, we say $a$ is {\em stable} in $p$, and we write $\mathbb{M}$ for the set $\{(a, p): \text{$a$ is stable in $p$}\}$.

\begin{lem} The relation $\equiv$ is an equivalence relation on $\mathbb{M}$.
\end{lem}

\begin{proof} Symmetry is clear, and reflexivity is immediate from the definition of $\mathbb{M}$. For transitivity, suppose $(a, p)\equiv (b, q)\equiv (c, r)$, and let $G_0\times G_1$ be $\mathbb{P}^2$-generic over $V$ with $p\in G_0, r\in G_1$; and let $H$ be $\mathbb{P}$-generic over $V[G_0\times G_1]$-generic, with $q\in H$. Then clearly in $V[G_0\times G_1][H]$, there is an isomorphism between $\nu[G_0]$ and $\nu[G_1]$ taking $a$ to $c$; but this is a $\Sigma^1_1$ property, and so already true in $V[G_0\times G_1]$. Thus, $(a, p)\equiv (c, r)$. \end{proof}

Now let $M$ be the set of $\equiv$-classes of elements of $\mathbb{M}$. The basic properties of $M$, which parallel the properties of ages needed for Fra\"iss\'e constructions, are:

\begin{lem}\label{EGA} For $p\in\mathbb{P}, a\in \omega$, 
\begin{enumerate}
\item {\bf (Extension)}  if $a$ is stable in $p$ and $q\le p$, then $a$ is stable in $q$ and $(a, p)\equiv (a, q)$; 
\item {\bf (Genericity)}  there is some $q\le p$ with $a$ stable in $q$; and 
\item {\bf (Amalgamation)}  for $(a_1, p_1), . . . , (a_n, p_n)\in M$, and $s\in\mathbb{P}$, there is some condition $r\in\mathbb{P}$ with $r\le s$ and $c_1, . . . c_n\in X$ such that $(c_1, r), . . . , (c_n, r)\in M$ and $(c_1, r)\equiv (a_1, p_1), . . . , (c_n, r)\equiv (a_n, p_n)$.
\end{enumerate}
\end{lem}

\begin{proof} $(1)$: That $a$ is stable in $q$ is immediate from the definition of stability. To see that $(a, p)\equiv (a, q)$, note that any pair of generics $H_0, H_1$ witnessing the failure of $(a, p)\equiv (a, q)$ would also witness the instability of $(a, p)$.

$(2)$: Consider the condition $(p, p)\in\mathbb{P}^2$. By our assumption on $\nu$, there must be some condition $(q, q')\le (p, p)$ and $a'\in X$ such that 
\[
(q, q')\Vdash_{\mathbb{P}^2} \{(a, a')\}\text{ extends to an isomorphism }\nu[\gl]\cong\nu[\gr].
\] 
It now follows that $a$ is stable in $q$: given $G_0\times G_1$ $\mathbb{P}^2$-generic over $V$ extending $(q,q)$, fix some $H$ which is $\mathbb{P}$-generic over $V[G_0\times G_1]$ with $q'\in H$. Thenm, clearly, in $V[G_0\times G_1][H]$ there is an isomorphism between $\nu[G_0]$ and $\nu[G_1]$ extending $\{(a, a)\}$; but this is a $\Sigma^1_1$ fact, so already true in $V[G_0\times G_1]$.

$(3)$: For simplicity, assume $s=\emptyset$ and $n=2$. Working in $\mathbb{P}^3$, by our assumption on $\nu$ we can find conditions $p_1', p_2', r'\in\mathbb{P}$ and $c_1, c_2\in X$ with $p_1'\le p_1, p_2'\le p_2,$ and \[(p_1', p_2', r')\Vdash_{\mathbb{P}^3}\text{$\{(a_1, c_1)\}$ and $\{(a_2, c_2)\}$ extend to $\nu[G_0]\cong\nu[G_2]$ and $\nu[G_1]\cong\nu[G_2]$,}\] respectively. Applying Lemma \ref{EGA}(2) twice to $r'$ then yields an $r\le r'$ with $c_1, c_2$ stable in $r$, so that $(c_1, r), (c_2, r)\in M$, and it is easily checked that $(c_1, r)\equiv (a_1, p_1)$ and $(c_2, r)\equiv (a_2, p_2)$.
\end{proof}

Finally, the following result is where rigidity is used. Intuitively, rigidity plays the role in our proof that $\omega$-homogeneity plays in standard Fra\"iss\'e limit constructions.

\begin{lem}\label{Sim} {\bf (Simultaneity)} Suppose $p, q\in\mathbb{P}$ and $i_1, . . . , i_n\colon\subseteq \omega\rightarrow \omega$ are partial maps in $V$ with disjoint domains which are each forced by $(p, q)$ in $\mathbb{P}^2$ to extend to isomorphisms $j_1, . . . , j_n\colon \nu[G_0]\cong\nu[G_1]$. Then \[(p, q)\Vdash_{\mathbb{P}^2}\bigcup_{1\le j\le n}i_j\text{ extends to an isomorphism }\nu[G_0]\cong\nu[G_1].\]
\end{lem} 

Note that this result immediately implies the seemingly stronger result in which disjointness of domains is not assumed.

\begin{proof} We will prove the lemma in the case where $n=2$, $p=q$, $i_1=\{(a, a)\}$ and $i_2=\{(b, b)\}$ for some distinct $a, b\in \om$; the general result is no different. Note that the assumption on $i_j$ in this case means just that $a$ and $b$ are stable in $p$.

Let $G_0\times G_1$ be $\mathbb{P}^2$-generic extending $(p, p)$. Then, forced by $(p, p)$, there are isomorphisms $j_1, j_2: \nu[G_0]\cong\nu[G_1]$ with $j_1(a)=a$ and $j_2(b)=b$. Consider the map $j=j_1\circ j_2^{-1}$. This is an automorphism of $\nu[G_1]$, and hence by rigidity must be the identity; so $j_1(b)=b$, since $j_2^{-1}(b)=b$ by assumption on $j_2$. But then $j_1$ is an isomorphism extending $\{(a, a), (b, b)\}$, so $(p, p)$ forces that there is an isomorphism between $\nu[G_0]$ and $\nu[G_1]$ extending $\{(a, a), (b, b)\}$.
\end{proof}

Now we come to the body of the proof of Theorem \ref{Rig}. We can turn $\mathbb{M}$ into an $\mathcal{L}$-structure, $\mathcal{M}$, as follows: writing $\bf(a, p)$ for the equivalence class of $(a, p)\in\mathbb{M}$, for each $n$-ary relation symbol $R\in\mathcal{L}$ we let $R^\mathcal{M}$ be the set of tuples (($\bf a_1, p_1$), . . . , ($\bf a_n, p_n$)) such that
\[ \exists q\in\mathbb{P}, c_1, . . . , c_n\text{ stable in $q$ }(\forall i\le n[\text{($\bf a_i, p_i$)= ($\bf c_i, q$)}]\,\,\wedge\,\, q\Vdash ``\nu\models R(c_1, . . . , c_n)").\]
Informally, this definition ensures that each relation $R$ holds whenever it ought to hold; we will also need the converse result, that each $R$ fails whenever it ought to fail, and this is where Simultaneity will come in.

\begin{lem} Let $G$ be $\mathbb{P}$-generic over $V$. Then $V[G]\models\nu[G]\cong\mathcal{M}$.
\end{lem}

\begin{proof} For $a\in\nu[G]$, let $Stab_a^G=\{p\in G: (a, p)\in \mathbb{M}\}$. Then for every $p, q\in Stab_a^G$, we must have $(a, p)\equiv(a, q)$: since $p, q\in G$, there must be a common strengthening $r\le p, q$; by \ref{EGA}(1), we have $(a, p)\equiv(a, r)$ and $(a, q)\equiv(a, r)$, and  hence $(a, p)\equiv (a, q)$ by transitivity. So the set $\{(a, p): p\in Stab_a^G\}$ is contained in a single $\equiv$-class, and hence corresponds to a single element of $\mathcal{M}$.

Consider the map $i\colon \nu[G]\rightarrow\mathcal{M}\colon a\mapsto \{(a, p): p\in Stab_a^G\}$; We claim that $i$ is an isomorphism. Surjectivity is an immediate consequence of Genericity \ref{EGA}(2), and injectivity follows from the rigidity of $\nu[G]$.

Finally, we must show that $i$ is a homomorphism. For $R$ a relation symbol in $\mathcal{L}$ and $\overline{a}\in\nu[G]$, by \ref{EGA}(3) we have \[(\nu[G]\models R(\overline{a}))\implies (\mathcal{M}\models R(i(\overline{a}))).\] Conversely, suppose $\mathcal{M}\models R(i(\overline{a}))$ and fix $p\in \bigcap_{a\in\overline{a}}Stab_a^G$. Then we must have some $\overline{c}$, $\overline{q}$, $\overline{b}$, and $r$ such that $(c_i, q_i)\equiv (a_i, p)$, $(c_i, q_i)\equiv (b_i, r)$, and $r\Vdash ``\nu\models R(\overline{b})$." But then by Simultaneity \ref{Sim}, since $(a_i, p)\equiv (b_i, r)$, we must have $p\Vdash R(\overline{a})$, and hence $\nu[G]\models R(\overline{a})$.
\end{proof}

This finishes the proof of Theorem \ref{Rig}.
\end{proof}


%

\begin{thebibliography}{GHHM13}

\bibitem[AK00]{AK00}
Chris~Ash and Julia~Knight.
\newblock {\em Computable structures and the hyperarithmetical hierarchy},
  volume 144 of {\em Studies in Logic and the Foundations of Mathematics}.
\newblock North-Holland Publishing Co., Amsterdam, 2000.

\bibitem[Bar73]{Bar73}
Jon Barwise.
\newblock Back and forth through infinitary logic.
\newblock In {\em Studies in model theory}, pages 5--34. MAA Studies in Math.,
  Vol. 8. Math. Assoc. Amer., Buffalo, N.Y., 1973.

\bibitem[Bar75]{Bar75}
Jon Barwise.
\newblock {\em Admissible sets and structures}.
\newblock Springer-Verlag, Berlin-New York, 1975.
\newblock An approach to definability theory, Perspectives in Mathematical
  Logic.

\bibitem[BFKL]{BFKL14}
John~Baldwin, Sy-David Friedman, Martin~Koerwien, and M.~Chris~Laskowski.
\newblock Three red herrings around Vaught's conjecture.
\newblock To appear.

\bibitem[CKM06]{CKM}
Wesley Calvert, Julia~Knight, and Jessica Millar.
\newblock Computable trees of {S}cott rank {$\omega_1^{CK}$}, and computable
  approximation.
\newblock {\em J. Symbolic Logic}, 71(1):283--298, 2006.

\bibitem[DPSS09]{DPSS}
Christian Delhomm{\'e}, Maurice Pouzet, G{\'a}bor S{\'a}gi, and Norbert Sauer.
\newblock Representation of ideals of relational structures.
\newblock {\em Discrete Math.}, 309(6):1374--1384, 2009.

\bibitem[Fra00]{Fra00}
Roland Fra{\"{\i}}ss{{\'e}}.
\newblock {\em Theory of relations}, volume 145 of {\em Studies in Logic and
  the Foundations of Mathematics}.
\newblock North-Holland Publishing Co., Amsterdam, revised edition, 2000.
\newblock With an appendix by Norbert Sauer.

\bibitem[GHHM13]{EMU}
Noam~Greenberg, Joel~David~Hamkins, Denis~Hirschfeldt, and Russell~Miller.
\newblock {\em Effective Mathematics of the Uncountable}.
\newblock Lecture Notes in Logic. Cambridge Univ, 2013.


\bibitem[Hod97]{shorter}
Wilfred~Hodges.
\newblock {\em A shorter model theory}.
\newblock Cambridge University Press, 1997.

\bibitem[Jec03]{Jec03}
Thomas Jech.
\newblock {\em Set theory}.
\newblock Springer Monographs in Mathematics. Springer-Verlag, Berlin, 2003.
\newblock The third millennium edition, revised and expanded.

\bibitem[Kar65]{Karp65}
Carol~Karp.
\newblock Finite-quantifier equivalence.
\newblock In {\em Theory of {M}odels ({P}roc. 1963 {I}nternat. {S}ympos.
  {B}erkeley)}, pages 407--412. North-Holland, Amsterdam, 1965.

\bibitem[Kei71]{Kei71}
H.~Jerome Keisler.
\newblock {\em Model theory for infinitary logic. {L}ogic with countable
  conjunctions and finite quantifiers}.
\newblock North-Holland Publishing Co., Amsterdam-London, 1971.
\newblock Studies in Logic and the Foundations of Mathematics, Vol. 62.

\bibitem[LS93]{SL93}
M.~Chris~Laskowski and Saharon~Shelah.
\newblock On the existence of atomic models.
\newblock {\em J. Symbolic Logic}, 58(4):1189--1194, 1993.

\bibitem[MM84]{MM84}
Angus Macintyre and David Marker.
\newblock Degrees of recursively saturated models.
\newblock {\em Trans. Amer. Math. Soc.}, 282(2):539--554, 1984.

\bibitem[Mon]{MonSR}
Antonio Montalb\'an.
\newblock A robuster Scott rank.
\newblock submitted for publication.

\bibitem[Mor70]{Mor70}
Michael Morley.
\newblock The number of countable models.
\newblock {\em J. Symbolic Logic}, 35:14--18, 1970.

\bibitem[MS08]{SM}
Jessica~Millar and Gerald~Sacks.
\newblock Atomic models higher up.
\newblock {\em Ann. Pure Appl. Logic}, 155(3):225--241, 2008.


\bibitem[Rab60]{Rab60}
Michael~Rabin.
\newblock Computable algebra, general theory and theory of computable fields.
\newblock {\em Trans. Amer. Math. Soc.}, 95:341--360, 1960.

\bibitem[Ric81]{Ric81}
Linda~Richter.
\newblock Degrees of structures.
\newblock {\em J. Symbolic Logic}, 46(4):723--731, 1981.

\bibitem[Sac07]{Sac07}
Gerald~Sacks.
\newblock Bounds on weak scattering.
\newblock {\em Notre Dame J. Formal Logic}, 48(1):5--31, 2007.

\bibitem[Sco65]{Sco65}
Dana~Scott.
\newblock Logic with denumerably long formulas and finite strings of
  quantifiers.
\newblock In {\em Theory of {M}odels ({P}roc. 1963 {I}nternat. {S}ympos.
  {B}erkeley)}, pages 329--341. North-Holland, Amsterdam, 1965.

\bibitem[Sol70]{Sol70}
Robert~Solovay.
\newblock A model of set-theory in which every set of reals is {L}ebesgue
  measurable.
\newblock {\em Ann. of Math. (2)}, 92:1--56, 1970.

\end{thebibliography}
%

\end{document}